\documentclass[twoside]{amsart}

\usepackage{amssymb}
\usepackage{amsfonts}
\usepackage{amsmath}
\usepackage{amsthm}
\usepackage{xcolor}
\usepackage{graphicx}
\usepackage{tikz}
\usepackage[margin={0.4in,0.8in}]{geometry}
\usepackage{tikz}
\usepackage{nopageno}
\usetikzlibrary{shapes.geometric}
\usetikzlibrary{quotes,angles}
\usepackage{float}
\usepackage{subcaption}
\usepackage{comment}

\usetikzlibrary{calc}

\DeclareMathOperator{\score}{score}
\DeclareMathOperator{\colve}{\operatorname{col_{ve}}}

\newtheorem{theorem}{Theorem}[section]

\newtheorem{corollary}[theorem]{Corollary}
\newtheorem{conj}[theorem]{Conjecture}

\theoremstyle{definition}
\newtheorem{definition}[theorem]{Definition}

\theoremstyle{remark}
\newtheorem{remark}[theorem]{Remark}

\tikzset{filled/.style={minimum width=5pt,inner sep=0pt,circle,fill=black}}
\tikzset{empty/.style={minimum width=5pt,inner sep=0pt,circle,fill=white, thick, draw=black}}
\tikzset{d/.style={minimum width=5pt,inner sep=0pt,circle,fill=black}}
\tikzset{dw/.style={minimum width=0,inner sep=0,circle,fill=black}}
\newcommand{\tikzAngleOfLine}{\tikz@AngleOfLine}
  \def\tikz@AngleOfLine(#1)(#2)#3{%
  \pgfmathanglebetweenpoints{%
    \pgfpointanchor{#1}{center}}{%
    \pgfpointanchor{#2}{center}}
  \pgfmathsetmacro{#3}{\pgfmathresult}%
  }

\numberwithin{equation}{section}
\title[Vertex Edge Marking Triangular Lattices]{Vertex-edge marking score of certain triangular lattices}
\author[]{Daniel Herden, Jonathan Meddaugh, Mark Sepanski, Isaac Echols, Nina Garcia-Montoya, Cordell Hammon, Guanjie Huang, Adam Kraus, Jorge Marchena Menendez, Jasmin Mohn, Rafael Morales Jiménez}
\address{
Department of Mathematics,
Baylor University,
Sid Richardson Building,
1410 S.4th Street,
Waco, TX 76706, USA}
\email{Daniel\_Herden@baylor.edu, Jonathan\_Meddaugh@baylor.edu, Mark\_Sepanski@baylor.edu}

\begin{document}
	

\keywords{vertex-edge marking game, coloring game, lattice, triangular lattice}
\subjclass[2020]{05C15, 05C57}

\begin{abstract}
	The vertex-edge marking game is played between two players on a graph, $G=(V,E)$, with one player marking vertices and the other marking edges. The players want to minimize/maximize, respectively, the number of marked edges incident to an unmarked vertex.
	The vertex-edge coloring number for $G$ is the maximum score achievable with perfect play.
	Bre\v{s}ar et al., \cite{bresar}, give an upper bound of $5$ for the vertex-edge coloring number for finite planar graphs. It is not known whether the bound is tight. In this paper, in response to questions in \cite{bresar}, we show that the vertex-edge coloring number for the infinite regular triangularization of the plane is 4. We also give two general techniques that allow us to calculate the vertex-edge coloring number in many related triangularizations of the plane.
\end{abstract}

\maketitle
\tableofcontents

\section{Introduction}

Combinatorial questions regarding colorings of maps and graphs go back to the 19th century, where Francis and Frederick Guthrie, under the advisement of De Morgan \cite{mckay}, posed the four-color conjecture. It was not until over a century later, in 1976, that a computer-assisted proof was presented by Appel and Haken \cite{suffice}.
Since that time, many variations of coloring problems have arisen. In particular, in 1981, Brams and Gardner \cite[Chapter 16, p. 253]{gardner} posed a coloring game on maps as a dynamical version of the map-coloring problem. As we neared the 21st century, Bodlaender \cite{complexity}, along with a number of other authors, began investigating similar coloring games. As introduced by Bartnicki et al. \cite{bartniki}, the \emph{vertex-edge marking game} is one of the many different coloring games that can be played on graphs.

The vertex-edge marking game involves two competing players: Alice marks vertices and Bob marks edges. Bob seeks to surround any unmarked vertex with as many marked edges as possible. Alice is competing against Bob and tries to limit the number of marked edges incident to any unmarked vertex. Starting with Alice, players alternate turns. On any turn of the game, the \emph{vertex score} at a vertex $v\in V$ is the number of marked edges incident to $v$ if $v$ is unmarked and $0$ otherwise. The \emph{final score} of the game is the maximum over all turns and vertices of the vertex score. For a given graph $G=(V,E)$, $1$ plus the final score of a game in which Alice and Bob each play optimally is called the \emph{vertex-edge coloring number} of $G$, denoted $\colve(G)$.

Bres\v{a}r et al. \cite{bresar} investigated many properties of the vertex-edge coloring number, including determining the upper bound $\colve(G) \le 5$ for all finite planar graphs. Multiple finite and infinite planar graphs have been found with $\colve(G) = 4$, but it remains unknown whether any exist with $\colve(G) = 5$.

Bres\v{a}r et al. \cite{bresar} expressed hope that possibly the infinte regular triangular lattice $\mathcal{T}$ might give an example of a graph $G$ with $\colve(G) = 5$. However, in Corollary \ref{Trianlge Lattice Theorem}, we show that $\colve (\mathcal T) = 4$. This follows from a more general result to bound $\colve(G)\le 4$ in Theorem \ref{Main Theorem} that depends on 2-colorability and angle markings. This result is applied to other triangular tilings of the plane in Corollaries \ref{Square Lattice Theorem} and \ref{subway tiling theorem}. A further technique for bootstrapping the calculation of $\colve(G)$ from a subgraph is given in Theorem \ref{graph extension theorem} and applied to graphs in Corollaries \ref{G'} and \ref{Extra Lines Theorem}. Section \ref{sec: further} ends with a technique to facilitate establishing higher lower bounds for $\colve(G)$ and some conjectures.

\section{Initial Definitions and Results} \label{sec: init defs}

The \textit{vertex-edge marking game} is played on a graph $G=(V,E)$ by two players who alternate turns each round of the game. At the beginning of the game, nothing on the graph is marked. The first player, known as Alice, marks vertices on her turn. The second player, known as Bob, marks edges on his turn. Alice's goal is to minimize the maximum number of marked edges adjacent to an unmarked vertex. Bob's goal is to maximize the maximum number of marked edges adjacent to an unmarked vertex.

More precisely, a game, $\mathcal G$, of the vertex-edge marking game played on the graph $G=(V,E)$ consists of a series of rounds, starting with round $r=1$, in which Alice and Bob each take a turn with Alice always going first. Alice always marks one of the remaining unmarked vertices and Bob always marks one of the remaining unmarked edges. At the end of round $r$, $r\geq 1$, write $\operatorname{MV}(r)$ and $\operatorname{UMV}(r)$ for the set of marked and unmarked vertices of $G$, respectively, and write $\operatorname{ME}(r)$ for the set of marked edges of $G$. For finite graphs, the game is played until either player runs out of moves, i.e., until either all vertices or all edges are marked. For an infinite graph, the game continues forever.

For $v\in V$, the \textit{vertex score} of $v$ after round $r$ is
\begin{equation*}
	\score(v, r) = \begin{cases}
		0 & \text{if } v \in \operatorname{MV}(r),\\
		|\{e \in \operatorname{ME}(r) : e \text{ is incident to } v\}|
							& \text{if } v \in \operatorname{UMV}(r).
	\end{cases}
\end{equation*}
The $r$\textit{-round score} after round $r$ is
\[ \score(r) = \sup \{\score(v,r) : v \in V\}.\]
The \textit{final score} of the game $\mathcal G$ is
\[ \score(\mathcal G) = \sup \{ \score(r) : \text{all rounds } r \}.\]

We say that Bob has a \textit{winning strategy} for the score $s$ if, regardless of Alice's moves each turn in a game $\mathcal G$, Bob can choose his moves to force $\score(\mathcal G) \geq s$. Finally,
the \textit{vertex-edge coloring number} of $G$ is
\[ \colve(G) = \sup \{ s : \text{Bob has a winning strategy for the score } s \} +1.\]
Note that, as long as $G$ has an edge, $\colve(G) \geq 2$. Clearly, $\colve(G) \leq \Delta(G) + 1$ where $\Delta(G)$ is the maximum degree of a vertex.

If $H$ is a subgraph of $G$, then \cite[Lemma 1]{bresar}
\begin{equation}\label{eqn:HinG}
	\colve(H) \leq \colve(G).
\end{equation}

A graph, $G$, has a $d$\textit{-bounded orientation} if the edges can be oriented so that each vertex has a maximal out-degree of $d$. In such a case, it is known \cite[Lemma 3]{bresar} that
\begin{equation}\label{eqn:outdegreebound}
	\colve(G) \leq d+2.
\end{equation}
From this it follows, \cite[Proposition 6]{bresar}, that for every finite planar graph, $G$,
\begin{equation}\label{eqn:planarleq5}
	\colve(G) \leq 5.
\end{equation}

In a vertex-edge marking game, $\mathcal G$, on a graph $G$, a  \textit{free path} is a path $P$ of $G$ with vertex sequence $v_0, \ldots, v_k$ with $k \geq 2$ so that
\begin{itemize}
	\item the first and last edges of $P$, $v_0v_1$ and $v_{k-1}v_k$, are marked,
	\item the interior vertices, $v_1,\ldots,v_{k-1}$, are not marked, and
	\item each interior vertex is incident to at least one edge not in $P$.
\end{itemize}
If there is a round in which it is Bob's turn and there is a free path in the game, then, \cite[Lemma 7]{bresar}, Bob has a strategy to force
\begin{equation}\label{freepath}
	\score(\mathcal G) \geq 3.
\end{equation}

\section{The Triangular Lattice}\label{sec: main results}

Write $\mathcal T$ for the infinite regular triangular lattice in the plane. In \cite[Question 3]{bresar}, the question was raised whether
$\colve(\mathcal T)$ is $4$ or $5$. In this section, we show that the answer is $\colve(\mathcal T) =4$.

\begin{theorem}\label{Main Theorem}
	Let $\mathcal L$ be a planar graph in the plane so that
	\begin{itemize}
		\item the bounded faces of $\mathcal L$ are $2$-colorable, gray and white,
		\item all gray faces of $\mathcal L$ are triangles,
        \item every edge of $\mathcal L$ belongs to precisely one gray triangle,
		\item exactly one angle from each gray triangle is marked, and
		\item each vertex of $\mathcal L$ is incident to at most two unmarked angles in gray triangles.
	\end{itemize}
Then \[ \colve(\mathcal{L}) \leq 4. \]
\end{theorem}

\begin{remark}
See Figures \ref{Trinagle Lattice Pic}, \ref{Square Lattice}, and \ref{Subway Tiling} for examples of infinite graphs that satisfy the conditions imposed on $\mathcal L$ here. Note that the white faces are not required to be triangles.
\end{remark}

\begin{proof}
Note that each edge of $\mathcal{L}$ belongs to exactly one gray triangle that we will call the \emph{corresponding triangle}. Therefore, as Bob marks edges, each marked edge is either the first, second, or third marked edge of its corresponding triangle.

To prove this theorem, we construct a strategy for Alice that only allows Bob to obtain a score of at most $3$ on any vertex. Alice begins by marking any vertex on her first turn. Alice's subsequent plays are determined by the edge marked by Bob in the previous round and whether it was the first, second, or third marked edge of its corresponding triangle, $T$. The rules are as follows.

\begin{itemize}
    \item[{\bf R1:}] If Bob marked the first edge of $T$, Alice marks the vertex incident to the marked angle in $T$, if the vertex is unmarked.

    If that vertex is already marked, Alice marks any other unmarked vertex in $T$, if one exists.

    Otherwise, Alice marks any unmarked vertex in $\mathcal{L}$.

    \item [{\bf R2:}] If Bob marked the second edge of $T$, Alice marks the vertex incident to both marked edges of $T$, if it is unmarked.

    If that vertex is already marked, Alice marks any other unmarked vertex in $T$, if one exists.

    Otherwise, Alice marks any unmarked vertex in $\mathcal{L}$.

    \item[{\bf R3:}] If Bob marked the third edge of $T$, Alice marks the remaining unmarked vertex in $T$, if one exists.

    Otherwise, Alice marks any unmarked vertex in $\mathcal{L}$.

\end{itemize}

We now show that this strategy prevents Bob from ever getting a score of $4$ or more on a vertex. We do this by showing that every vertex incident to (exactly) three marked edges at the end of Bob's turn is already marked or will be marked by Alice at the start of the next round.

Suppose $v$ is an unmarked vertex incident to three marked edges at the end of a round (after Bob's turn). In the order that Bob played them, label these three edges $e_1, e_2, e_3$.

Consider first the case where there is an $e_i$ whose corresponding gray triangle, $S$, has a marked angle incident to $v$. As $v$ is unmarked, rule R1 implies that $i=3$, that Bob just played $e_3$, and that Alice will mark $v$ on her next turn.

We may now assume that all marked edges belong to corresponding gray triangles whose marked vertices are not incident to~$v$. As there are at most two gray triangles with unmarked angles incident to $v$, at least two of the marked edges belong to the same corresponding triangle, $S$. If all three edges of $S$ are marked after Bob's turn, then rule R3 implies that $e_3$ is an edge of $S$, that Bob just played $e_3$, and that Alice will mark $v$ on her next turn. However, if the third edge of $S$ is unmarked, then rule R2 implies that $e_3$ is an edge of $S$, that Bob just played $e_3$, and that Alice will mark $v$ on her next turn.
\end{proof}

\begin{corollary}\label{Trianlge Lattice Theorem}
    Let $\mathcal{T}$ be the infinite regular triangular lattice in the plane. Then \[ \colve(\mathcal{T}) = 4. \]
\end{corollary}
\begin{proof}
	Since it is known that $\colve(\mathcal H) = 4$ for the infinite regular hexagonal lattice $\mathcal H$, \cite[Theorem 10]{bresar}, the lower bound, $\colve(\mathcal{T}) \geq 4$, follows from Inequality \eqref{eqn:HinG}.

We use Theorem \ref{Main Theorem} to obtain the upper bound.
Begin by $2$-coloring the faces of $\mathcal{T}$ as in Figure \ref{Trinagle Lattice Pic} so that the gray triangles point to the right and the white triangles point to the left. Mark the unique rightmost angle of each gray triangle, the one pointing to the right. Theorem \ref{Main Theorem} shows that $\colve(\mathcal{T})\leq 4$.
\end{proof}

\begin{figure}[hbt!]
\begin{tikzpicture}[scale = 0.5]
    \def \zzz{1cm}
    \def \hhh{1.732cm}

    \foreach \j in {1,3,5,7,9}{
        \foreach \i in {0,...,3}{
            \filldraw[fill=gray!60!white!,thick, draw =black] (\i*\hhh*2+\hhh,\j*\zzz) -- (\i*\hhh*2,\j*\zzz+\zzz) -- (\i*\hhh*2,\j*\zzz-\zzz) -- (\i*\hhh*2+\hhh,\j*\zzz);
            \filldraw[fill=gray!60!white!,thick, draw =black] (\i*\hhh*2+\hhh,\j*\zzz) -- (\i*\hhh*2+\hhh,\j*\zzz+2*\zzz) -- (\i*\hhh*2+2*\hhh,\j*\zzz+\zzz) -- (\i*\hhh*2+\hhh,\j*\zzz);
        }
    }

        \foreach \j in {1,3,5,7,9,11}{
        \foreach \i in {0,...,3}{
            \node[fill,circle, scale = .1] (\i\j) at (\i*\hhh*2+\hhh,\j*\zzz) {};
        }
    }
    \foreach \j in {0,2,4,6,8,10}{
        \foreach \i in {0,...,4}{
            \node[fill,circle, scale = .1] (\i\j) at (\i*\hhh*2,\j*\zzz) {};
        }
    }

    \draw[black,  thick] (40) -- (410);
    \draw[black,  thick] (010) -- (011);
    \draw[black,  thick] (110) -- (111);
    \draw[black,  thick] (210) -- (211);
    \draw[black,  thick] (310) -- (311);
    \draw[black,  thick] (01) -- (10);
    \draw[black,  thick] (11) -- (20);
    \draw[black,  thick] (21) -- (30);
    \draw[black,  thick] (31) -- (40);

    \draw pic[draw=black, thick,-,angle eccentricity=1.2,angle radius=.35cm] {angle=02--01--00};
    \draw pic[draw=black, thick,-,angle eccentricity=1.2,angle radius=.35cm] {angle=12--11--10};
    \draw pic[draw=black, thick,-,angle eccentricity=1.2,angle radius=.35cm] {angle=22--21--20};
    \draw pic[draw=black,thick,-,angle eccentricity=1.2,angle radius=.35cm] {angle=32--31--30};
    \draw pic[draw=black,thick,-,angle eccentricity=1.2,angle radius=.35cm] {angle=04--03--02};
    \draw pic[draw=black,thick,-,angle eccentricity=1.2,angle radius=.35cm] {angle=14--13--12};
    \draw pic[draw=black,thick,-,angle eccentricity=1.2,angle radius=.35cm] {angle=24--23--22};
    \draw pic[draw=black,thick,-,angle eccentricity=1.2,angle radius=.35cm] {angle=34--33--32};
    \draw pic[draw=black,thick,-,angle eccentricity=1.2,angle radius=.35cm] {angle=06--05--04};
    \draw pic[draw=black,thick,-,angle eccentricity=1.2,angle radius=.35cm] {angle=16--15--14};
    \draw pic[draw=black,thick,-,angle eccentricity=1.2,angle radius=.35cm] {angle=26--25--24};
    \draw pic[draw=black,thick,-,angle eccentricity=1.2,angle radius=.35cm] {angle=36--35--34};
    \draw pic[draw=black,thick,-,angle eccentricity=1.2,angle radius=.35cm] {angle=08--07--06};
    \draw pic[draw=black,thick,-,angle eccentricity=1.2,angle radius=.35cm] {angle=18--17--16};
    \draw pic[draw=black,thick,-,angle eccentricity=1.2,angle radius=.35cm] {angle=28--27--26};
    \draw pic[draw=black,thick,-,angle eccentricity=1.2,angle radius=.35cm] {angle=38--37--36};
    \draw pic[draw=black,thick,-,angle eccentricity=1.2,angle radius=.35cm] {angle=010--09--08};
    \draw pic[draw=black,thick,-,angle eccentricity=1.2,angle radius=.35cm] {angle=110--19--18};
    \draw pic[draw=black,thick,-,angle eccentricity=1.2,angle radius=.35cm] {angle=210--29--28};
    \draw pic[draw=black,thick,-,angle eccentricity=1.2,angle radius=.35cm] {angle=310--39--38};

    \draw pic[draw=black,thick,-,angle eccentricity=1.2,angle radius=.35cm] {angle=03--12--01};
    \draw pic[draw=black,thick,-,angle eccentricity=1.2,angle radius=.35cm] {angle=05--14--03};
    \draw pic[draw=black,thick,-,angle eccentricity=1.2,angle radius=.35cm] {angle=07--16--05};
    \draw pic[draw=black,thick,-,angle eccentricity=1.2,angle radius=.35cm] {angle=09--18--07};
    \draw pic[draw=black,thick,-,angle eccentricity=1.2,angle radius=.35cm] {angle=011--110--08};
    \draw pic[draw=black,thick,-,angle eccentricity=1.2,angle radius=.35cm] {angle=13--22--11};
    \draw pic[draw=black,thick,-,angle eccentricity=1.2,angle radius=.35cm] {angle=15--24--13};
    \draw pic[draw=black,thick,-,angle eccentricity=1.2,angle radius=.35cm] {angle=17--26--15};
    \draw pic[draw=black,thick,-,angle eccentricity=1.2,angle radius=.35cm] {angle=19--28--17};
    \draw pic[draw=black,thick,-,angle eccentricity=1.2,angle radius=.35cm] {angle=111--210--18};
    \draw pic[draw=black,thick,-,angle eccentricity=1.2,angle radius=.35cm] {angle=23--32--21};
    \draw pic[draw=black,thick,-,angle eccentricity=1.2,angle radius=.35cm] {angle=25--34--23};
    \draw pic[draw=black,thick,-,angle eccentricity=1.2,angle radius=.35cm] {angle=27--36--25};
    \draw pic[draw=black,thick,-,angle eccentricity=1.2,angle radius=.35cm] {angle=29--38--27};
    \draw pic[draw=black,thick,-,angle eccentricity=1.2,angle radius=.35cm] {angle=211--310--28};
    \draw pic[draw=black,thick,-,angle eccentricity=1.2,angle radius=.35cm] {angle=33--42--31};
    \draw pic[draw=black,thick,-,angle eccentricity=1.2,angle radius=.35cm] {angle=35--44--33};
    \draw pic[draw=black,thick,-,angle eccentricity=1.2,angle radius=.35cm] {angle=37--46--35};
    \draw pic[draw=black,thick,-,angle eccentricity=1.2,angle radius=.35cm] {angle=39--48--37};
    \draw pic[draw=black,thick,-,angle eccentricity=1.2,angle radius=.35cm] {angle=311--410--38};

\end{tikzpicture}
\caption{$\mathcal T$ With Marked Angles}
\label{Trinagle Lattice Pic}
\end{figure}
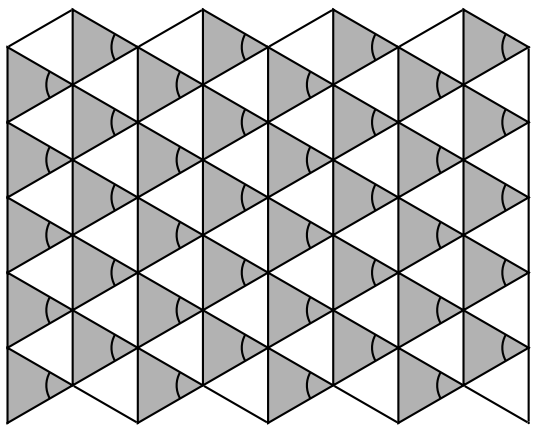

\section{Other Triangularizations}

Theorem \ref{Main Theorem} applies to many graphs. For example, write $\mathcal{R}$ for the infinite triangular lattice obtained by adding a vertex to the center of each face of the infinite regular square lattice with added edges between each new vertex and each vertex in the corresponding face. See Figure \ref{Square Lattice}.

\begin{figure}[H]
	\scalebox{.5}{
		\begin{tikzpicture}[
			right angle triangle/.style={
				isosceles triangle,
				isosceles triangle apex angle=90},
			every node/.style={right angle triangle,
				draw, inner sep=0pt,
				anchor=left corner,
				shape border rotate=90},
			paint/.style={draw=black, fill=#1!60}
			]

			\begin{scope}[yshift=-3cm,
				every node/.style={
					right angle triangle,
					isosceles triangle stretches=false,
					draw, inner sep=0pt,
					minimum height=1cm,
					anchor=apex}
				]
				\draw (0,0) grid (8,8);
				\foreach \i in {0,2}
				\foreach \j in {0,2}
				\foreach \k in {1,3}
				\foreach \a/\c in {45/white, 135/gray,225/white,315/gray}
				\node[paint=\c,
				shape border rotate=\a] at (2*\i + \k,2*\j+\k) {};
				\foreach \i in {1,3}
				\foreach \j in {0,2}
				\foreach \a/\c in {45/gray, 135/white,225/gray,315/white}
				\node[paint=\c,
				shape border rotate=\a] at (2*\i +1 ,2*\j +1) {};
				\foreach \i in {1,3}
				\foreach \j in {0,2}
				\foreach \a/\c in {45/gray, 135/white,225/gray,315/white}
				\node[paint=\c,
				shape border rotate=\a] at (2*\i-1  ,2*\j +3) {};
				\foreach \i in {0,2,4,6,8}
				\draw[black, thick] (\i,0) -- (\i,8);
				\foreach \i in {0,2,4,6,8}
				\draw[black, thick] (0,\i) -- (8,\i);
				
			\end{scope}
	\end{tikzpicture}}
	\caption{Triangularization $\mathcal R$}
	\label{Square Lattice}
\end{figure}
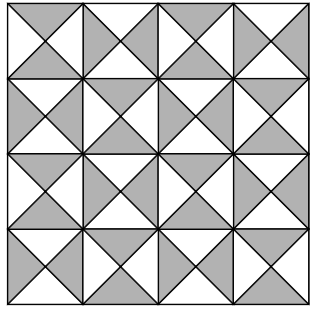

\begin{corollary}\label{Square Lattice Theorem}
 $\colve(\mathcal{R}) = 4$.
\end{corollary}

\begin{proof}
	
	Since the infinite square lattice $\mathcal S$ has $\colve(\mathcal S) = 4$, \cite[Proposition 11]{bresar}, the lower bound, $\colve(\mathcal{R}) \geq 4$, follows from Inequality~\eqref{eqn:HinG}.
	
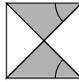
\begin{figure}[hbt!]
\scalebox{.5}{
\begin{tikzpicture}
\node (A) at (0,0) {};
\node (B) at (1,1) {};
\node (C) at (2,0) {};
\node (D) at (2,2) {};
\node (E) at (0,2) {};
\filldraw[fill=gray!60!white!,line width=0.8pt] (0,0) -- (1,1) -- (2,0) --cycle;
\draw [line width=0.8pt](0,0) -- (1,1) -- (0,2) --cycle;
\filldraw[fill=gray!60!white!,line width=0.8pt] (0,2) -- (1,1) -- (2,2) --cycle;
\draw [line width=0.8pt](2,0) -- (1,1) -- (2,2) --cycle;
\draw pic[draw=black, thick,-,angle eccentricity=1.2,angle radius=.7cm]
 {angle = B--C--A};
\draw pic[draw=black, thick,-,angle eccentricity=1.2,angle radius=.7cm]
 {angle = E--D--B};
\end{tikzpicture}
}
\caption{Vertical Triangles}
\label{Single Square With Right Angles Marked}
\end{figure}


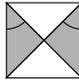
\begin{figure}[hbt!]
\scalebox{.5}{
\begin{tikzpicture}
\node (A) at (0,0) {};
\node (B) at (1,1) {};
\node (C) at (2,0) {};
\node (D) at (2,2) {};
\node (E) at (0,2) {};
\filldraw[fill=gray!60!white!,line width=0.8pt] (0,0) -- (1,1) -- (0,2) --cycle;
\draw [line width=0.8pt](0,0) -- (1,1) -- (2,0) --cycle;
\filldraw[fill=gray!60!white!,line width=0.8pt] (2,0) -- (1,1) -- (2,2) --cycle;
\draw [line width=0.8pt](0,2) -- (1,1) -- (2,2) --cycle;
\draw pic[draw=black, thick,-,angle eccentricity=1.2,angle radius=.7cm]
 {angle = A--E--B};
\draw pic[draw=black, thick,-,angle eccentricity=1.2,angle radius=.7cm]
 {angle = B--D--C};
\end{tikzpicture}
}
\caption{Horizontal Triangles}
\label{Singe Square With Up Angles Marked}
\end{figure}


	To get the upper bound, begin by $2$-coloring the faces of $\mathcal R$ as in Figure \ref{Square Lattice}. Notice the gray triangles are either vertical, as in Figure \ref{Single Square With Right Angles Marked}, or horizontal, as in Figure \ref{Singe Square With Up Angles Marked}. Mark the vertical and horizontal gray triangles as pictured in Figures \ref{Single Square With Right Angles Marked} and \ref{Singe Square With Up Angles Marked}, respectively, with vertical triangles marked to the right and horizontal triangles marked upwards. Theorem \ref{Main Theorem} shows that $\colve(\mathcal{R})\leq 4$.
\end{proof}

As another example, write $\mathcal C$ for
 the infinite triangular lattice $\mathcal{C}$ obtained from the regular square-octagon lattice by adding a vertex in the center of each face and edges from each new vertex to the vertices of the corresponding face. See Figure \ref{Subway Tiling}.

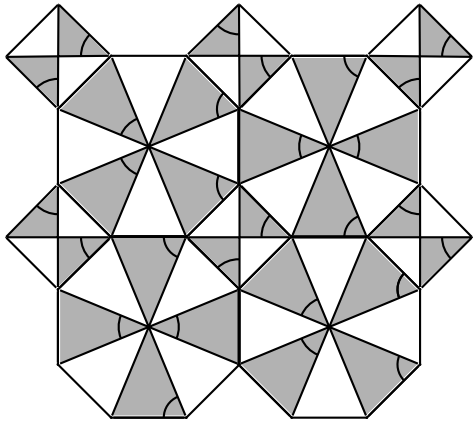
\begin{figure}[H]
	\begin{tikzpicture}[scale = 1]
		\def\zzz{1cm}
		
		\draw[black, thick] (-1.7,0) foreach \x[count=\xi from 0] in {A,...,H}{
			node[dw] (\x){}
			--++(\xi*45:\zzz)
		};
		
		\filldraw[fill=gray!60!white!,line width=0.8pt, draw =black] (G) -- (F) -- (intersection of A--E and B--F) -- (G);
		\filldraw[fill=gray!60!white!,line width=0.8pt, draw =black] (E) -- (D) -- (intersection of A--E and B--F) -- (E);
		\filldraw[fill=gray!60!white!,line width=0.8pt, draw =black] (H) -- (A) -- (intersection of A--E and B--F) -- (H);
		\filldraw[fill=gray!60!white!,line width=0.8pt, draw =black] (B) -- (C) -- (intersection of A--E and B--F) -- (B);
		\node[filled, scale=.1] (C1) at (intersection of A--E and B--F) {};
		\node[filled, scale=.1] at (A) {};
		\node[filled, scale=.1] at (B) {};
		\node[filled, scale=.1] at (C) {};
		\node[filled, scale=.1] at (D) {};
		\node[filled, scale=.1] at (E) {};
		\node[filled, scale=.1] at (F) {};
		\node[filled, scale=.1] at (G) {};
		\node[filled, scale=.1] at (H) {};		
		
		\draw[black, thick] (0.7,0) foreach \x[count=\xi from 0] in {I,...,P}{
			node[dw] (\x){}
			--++(\xi*45:\zzz)
		};
		
		\filldraw[fill=gray!60!white!,line width=0.8pt, draw =black] (I) -- (J) -- (intersection of I--M and J--N) -- (I);
		\filldraw[fill=gray!60!white!,line width=0.8pt, draw =black] (K) -- (L) -- (intersection of I--M and J--N) -- (K);
		\filldraw[fill=gray!60!white!,line width=0.8pt, draw =black] (M) -- (N) -- (intersection of I--M and J--N) -- (M);
		\filldraw[fill=gray!60!white!,line width=0.8pt, draw =black] (O) -- (P) -- (intersection of I--M and J--N) -- (O);
		\node[filled, scale=.1] (C2) at (intersection of I--M and J--N) {};
		\node[filled, scale=.1] at (I) {};
		\node[filled, scale=.1] at (J) {};
		\node[filled, scale=.1] at (K) {};
		\node[filled, scale=.1] at (L) {};
		\node[filled, scale=.1] at (M) {};
		\node[filled, scale=.1] at (N) {};
		\node[filled, scale=.1] at (O) {};
		\node[filled, scale=.1] at (P) {};
		
		\draw[black, thick] (-1.7,-2.4) foreach \x[count=\xi from 0] in {Q,...,X}{
			node[dw] (\x){}
			--++(\xi*45:\zzz)
		};

		\filldraw[fill=gray!60!white!,line width=0.8pt, draw =black] (Q) -- (R) -- (intersection of R--V and Q--U) -- (Q);
		\filldraw[fill=gray!60!white!,line width=0.8pt, draw =black] (S) -- (T) -- (intersection of R--V and Q--U) -- (S);
		\filldraw[fill=gray!60!white!,line width=0.8pt, draw =black] (U) -- (V) -- (intersection of R--V and Q--U) -- (U);
		\filldraw[fill=gray!60!white!,line width=0.8pt, draw =black] (W) -- (X) -- (intersection of R--V and Q--U) -- (W);
		\node[filled, scale=.1] (C3) at (intersection of R--V and Q--U) {};
		\node[filled, scale=.1] at (Q) {};
		\node[filled, scale=.1] at (R) {};
		\node[filled, scale=.1] at (S) {};
		\node[filled, scale=.1] at (T) {};
		\node[filled, scale=.1] at (U) {};
		\node[filled, scale=.1] at (V) {};
		\node[filled, scale=.1] at (W) {};
		\node[filled, scale=.1] at (X) {};
		
		\draw[black, thick] (0.7,-2.4) foreach \x[count=\xi from 0] in {Q,...,X}{
			node[dw] (\x\x){}
			--++(\xi*45:\zzz)
		};

		\filldraw[fill=gray!60!white!,line width=0.8pt, draw =black] (SS) -- (RR) -- (intersection of RR--VV and QQ--UU) -- (SS);
		\filldraw[fill=gray!60!white!,line width=0.8pt, draw =black] (UU) -- (TT) -- (intersection of RR--VV and QQ--UU) -- (UU);
		\filldraw[fill=gray!60!white!,line width=0.8pt, draw =black] (WW) -- (VV) -- (intersection of RR--VV and QQ--UU) -- (WW);
		\filldraw[fill=gray!60!white!,line width=0.8pt, draw =black] (QQ) -- (XX) -- (intersection of RR--VV and QQ--UU) -- (QQ);
		\node[filled, scale=.1] (C4) at (intersection of RR--VV and QQ--UU) {};
		\node[filled, scale=.1] at (QQ) {};
		\node[filled, scale=.1] at (RR) {};
		\node[filled, scale=.1] at (SS) {};
		\node[filled, scale=.1] at (TT) {};
		\node[filled, scale=.1] at (UU) {};
		\node[filled, scale=.1] at (VV) {};
		\node[filled, scale=.1] at (WW) {};
		\node[filled, scale=.1] at (XX) {};		
		
		\node[filled, scale=.1] (AA) at (0,3.1){};
		\filldraw[fill=gray!60!white!,line width=0.8pt, draw =black] (E) -- (AA) -- (intersection of AA--D and E--N) -- (E);
		\node[filled, scale=.1] () at (0,3.1){};
		\filldraw[fill=gray!60!white!,line width=0.8pt, draw =black] (N) -- (D) -- (intersection of AA--D and E--N) -- (N);
		\draw[black, thick] (AA) -- (N);
		\node[filled, scale=.1] (S1) at (intersection of AA--D and E--N) {};
		
		\node[filled, scale=.1] (BB) at (2.4,3.1){};
		\node[filled, scale=.1] (CC) at (3.1,2.4){};
		\filldraw[fill=gray!60!white!,line width=0.8pt, draw =black] (M) -- (intersection of BB--L and CC--M) -- (L) -- (M);
		\filldraw[fill=gray!60!white!,line width=0.8pt, draw =black] (BB) -- (intersection of BB--L and CC--M) -- (CC) -- (BB);

		\draw[black, thick] (BB) -- (M);
		\draw[black, thick] (CC) -- (L);
		\node[filled, scale=.1] (S2) at (intersection of BB--L and CC--M) {};

		\node[filled, scale=.1] (BBB) at (-2.4,3.1){};
		\node[filled, scale=.1] (CCC) at (-3.1,2.4){};	
		\filldraw[fill=gray!60!white!,line width=0.8pt, draw =black] (CCC) -- (intersection of BBB--G and CCC--F) -- (G) -- (CCC);
		\filldraw[fill=gray!60!white!,line width=0.8pt, draw =black] (BBB) -- (intersection of BBB--G and CCC--F) -- (F) -- (BBB);
		\draw[black, thick] (BBB) -- (CCC);
		
		\filldraw[fill=gray!60!white!,line width=0.8pt, draw =black] (C) -- (I) -- (intersection of C--T and B--I) -- (C);
		\filldraw[fill=gray!60!white!,line width=0.8pt, draw =black] (B) -- (T) -- (intersection of C--T and B--I) -- (B);
		\node[filled, scale=.1] (S3) at (intersection of C--T and B--I) {};
		
		\node[filled, scale=.1] (DD) at (3.1,0){};
		\filldraw[fill=gray!60!white!,line width=0.8pt, draw =black] (J) -- (K) -- (intersection of J--DD and K--TT) -- (J);
		\filldraw[fill=gray!60!white!,line width=0.8pt, draw =black] (DD) -- (TT) -- (intersection of J--DD and K--TT) -- (DD);
		\node[filled, scale=.1] (DD) at (3.1,0){};
		\draw[black, thick] (DD) -- (K);
		\node[filled, scale=.1] (S4) at (intersection of J--DD and K--TT) {};	

		\node[filled, scale=.1] (DDD) at (-3.1,0){};	
		\filldraw[fill=gray!60!white!,line width=0.8pt, draw =black] (DDD) -- (intersection of DDD--A and H--W) -- (H) -- (DDD);
		\filldraw[fill=gray!60!white!,line width=0.8pt, draw =black] (W) -- (intersection of DDD--A and H--W) -- (A) -- (W);
		\draw[black, thick] (DDD) -- (W);
		
		\tikzAngleOfLine(C1)(H){\AngleStart}
		\tikzAngleOfLine(C1)(A){\AngleEnd}
		\draw[black, thick, -] (C1)+(\AngleStart:.4cm) arc (\AngleStart:\AngleEnd:.4 cm);
		
		\tikzAngleOfLine(C1)(G){\AngleStart}
		\tikzAngleOfLine(C1)(F){\AngleEnd}
		\draw[black, thick, -] (C1)+(\AngleStart:.4cm) arc (\AngleStart:\AngleEnd:.4 cm);
		
		\tikzAngleOfLine(D)(C1){\AngleStart}
		\tikzAngleOfLine(D)(E){\AngleEnd}
		\draw[black, thick, -] (D)+(\AngleStart:.3cm) arc (\AngleStart:\AngleEnd:.3 cm);
		
		\tikzAngleOfLine(C)(C1){\AngleStart}
		\tikzAngleOfLine(C)(B){\AngleEnd}
		\draw[black, thick, -] (C)+(\AngleStart:.3cm) arc (\AngleStart:\AngleEnd:.3 cm);
		
		\tikzAngleOfLine(AA)(E){\AngleStart}
		\tikzAngleOfLine(AA)(S1){\AngleEnd}
		\draw[black, thick, -] (AA)+(\AngleStart:.4cm) arc (\AngleStart:\AngleEnd:.4 cm);
		
		\tikzAngleOfLine(N)(S1){\AngleStart}
		\tikzAngleOfLine(N)(D){\AngleEnd}
		\draw[black, thick, -] (N)+(\AngleStart:.4cm) arc (\AngleStart:\AngleEnd:.4 cm);
		
		\tikzAngleOfLine(C2)(O){\AngleStart}
		\tikzAngleOfLine(C2)(P){\AngleEnd}
		\draw[black, thick, -] (C2)+(\AngleStart:.4cm) arc (\AngleStart:\AngleEnd:.4 cm);
		
		
		\draw pic[draw=black,thick,-,angle eccentricity=1.2,angle radius=.4cm] {angle=K--C2--L};
		\draw pic[draw=black,thick,-,angle eccentricity=1.2,angle radius=.4cm] {angle=S2--L--M};
		\draw pic[draw=black,thick,-,angle eccentricity=1.2,angle radius=.4cm] {angle=BB--CC--S2};
		\draw pic[draw=black,thick,-,angle eccentricity=1.2,angle radius=.4cm] {angle=W--C3--X};
		\draw pic[draw=black,thick,-,angle eccentricity=1.2,angle radius=.4cm] {angle=S--C3--T};
		\draw pic[draw=black,thick,-,angle eccentricity=1.2,angle radius=.4cm] {angle=VV--C4--WW};
		\draw pic[draw=black,thick,-,angle eccentricity=1.2,angle radius=.4cm] {angle=XX--C4--QQ};
		\draw pic[draw=black,thick,-,angle eccentricity=1.2,angle radius=.4cm] {angle=P--VV--S3};
		\draw pic[draw=black,thick,-,angle eccentricity=1.2,angle radius=.4cm] {angle=S3--WW--U};
		\draw pic[draw=black,thick,-,angle eccentricity=1.2,angle radius=.4cm] {angle=J--K--S4};
		\draw pic[draw=black,thick,-,angle eccentricity=1.2,angle radius=.4cm] {angle=S4--DD--TT};
		\draw pic[draw=black,thick,-,angle eccentricity=1.2,angle radius=.3cm] {angle=C4--SS--RR};
		\draw pic[draw=black,thick,-,angle eccentricity=1.2,angle radius=.3cm] {angle=UU--TT--C4};
		\draw pic[draw=black,thick,-,angle eccentricity=1.2,angle radius=.3cm] {angle=N--M--C2};
		\draw pic[draw=black,thick,-,angle eccentricity=1.2,angle radius=.3cm] {angle=C2--J--I};
		\draw pic[draw=black,thick,-,angle eccentricity=1.2,angle radius=.3cm] {angle=UU--TT--C4};
		\draw pic[draw=black,thick,-,angle eccentricity=1.2,angle radius=.3cm] {angle=V--U--C3};
		\draw pic[draw=black,thick,-,angle eccentricity=1.2,angle radius=.3cm] {angle=C3--R--Q};		

		\draw pic[draw=black,thick,-,angle eccentricity=1.2,angle radius=.4cm] {angle=BBB--G--CCC};
		\draw pic[draw=black,thick,-,angle eccentricity=1.2,angle radius=.4cm] {angle=BBB--F--CCC};	
		\draw pic[draw=black,thick,-,angle eccentricity=1.2,angle radius=.4cm] {angle=DDD--H--W};
		\draw pic[draw=black,thick,-,angle eccentricity=1.2,angle radius=.4cm] {angle=DDD--A--W};
		
	\end{tikzpicture}
	\caption{Triangularization $\mathcal C$}
	\label{Subway Tiling}
\end{figure}

\begin{corollary} \label{subway tiling theorem}
 $\colve(\mathcal{C}) = 4$.
\end{corollary}

\begin{proof}
	As in Corollary \ref{Square Lattice Theorem}, the infinite square lattice $\mathcal S$ is a subgraph of $\mathcal C$, and $\colve(\mathcal C) \geq 4$ holds. For an upper bound, color the faces of $\mathcal C$ and mark angles as in Figure \ref{Subway Tiling}. Then Theorem \ref{Main Theorem} shows that $\colve(\mathcal{C})\leq 4$.
\end{proof}

Certain geometric structures permit the extension of a bound for $\colve(G)$ to a larger graph $G'$ if we know enough about Alice's strategy. As a starting illustration, let $\mathcal T'$ be $\mathcal T$ with a vertex added to the center of each face and edges connecting each new vertex to the vertices of the corresponding face.
See Figure \ref{Triangle Lattice With Extra Lines}. Note that the resulting graph is not $2$-colorable.

\begin{figure}[hbt!]
	\begin{tikzpicture}[scale=1]
		\clip       (-1, {-1/sqrt(3)}) rectangle (2.73, 2.8);
		
		\draw[black,thick] (2.72,2.8)--(2.72, -0.57);
		
		\foreach \i in {0,...,2}{
			\foreach \j in {0,...,2}{
				\node[isosceles triangle,
				isosceles triangle apex angle=60,
				draw,fill=gray!60,thick,  minimum size =1cm] at (\i*2,{\j/sin(60)}){};
		}}
		\foreach \j in {0,...,3}{
			\node[isosceles triangle,
			isosceles triangle apex angle=60,
			draw,fill=gray!60,thick,  minimum size =1cm] at (1,{(\j-1)/sin(60)+tan(30)}){};
		}
		
		\foreach \i in {0,2}{
			\foreach \j in {0,...,2}{
				
				
				\draw[black,thick] (\i,{\j/sin(60)})--({\i-0.3}, {\j/sin(60)+1/sqrt(3)});
				\draw[black,thick] (\i,{\j/sin(60)})--({\i-0.3}, {(\j-1)/sin(60)+1/sqrt(3)});
				\draw[black,thick] (\i,{\j/sin(60)})--({\i+0.7}, {\j/sin(60)});
				\draw[black,thick] ({\i+0.7}, {\j/sin(60)})--({\i+0.4}, {\j/sin(60)+1/sqrt(3)});
				\draw[black,thick] ({\i+0.7}, {\j/sin(60)})--({\i+0.4}, {(\j-1)/sin(60)+1/sqrt(3)});
				\draw[black,thick] ({\i-0.3},{(\j-1)/sin(60)+1/sqrt(3)})--({\i+0.4}, {(\j-1)/sin(60)+1/sqrt(3)});
		}}

		\foreach \j in {0,...,2}{
			
			\draw[black,thick] ({1}, {(\j+1)/sin(60)-1/sqrt(3)})--({1+0.7}, {(\j+1)/sin(60)-1/sqrt(3)});
			\draw[black,thick] ({.7}, {\j/sin(60)})--({1.4}, {\j/sin(60)});
			
			\draw[black,thick] ({1}, {(\j+1)/sin(60)-1/sqrt(3)})--({.7}, {\j/sin(60)});
			\draw[black,thick] ({1},{\j/sin(60)-1/sqrt(3)})--({.7}, {\j/sin(60)});
			\draw[black,thick] ({1+.7}, {(\j+1)/sin(60)-1/sqrt(3)})--({1.4}, {\j/sin(60)});
			\draw[black,thick] ({1.7},{\j/sin(60)-1/sqrt(3)})--({1.4}, {\j/sin(60)});
			
		}

	\end{tikzpicture}
	\caption{Triangularization $\mathcal T'$}
	\label{Triangle Lattice With Extra Lines}
\end{figure}
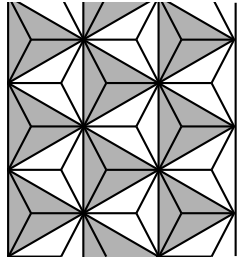

\begin{theorem}\label{graph extension theorem}
	Let $G'=(V',E')$ be a graph with a vertex-induced subgraph $G=(V,E)$ so that
	\begin{itemize}
		\item for $v\in V' -V$, $\deg(v) < n$ in $G'$,
		\item $\colve(G) = n$, and
		\item Alice has a strategy on $G$ able to leave no unmarked vertices with $n-1$ marked incident edges at the end of her turn.
	\end{itemize}
    Then \[\colve(G')= n.\]
\end{theorem}

\begin{proof}
	We begin with a note on the third requirement. To prevent Bob from getting a score of $n$ on $G$, Alice can ignore vertices in $G$ that only have degree $n-1$. In other words, she only needs a strategy able to mark any vertex with $n-1$ marked incident edges \textit{and at least one other unmarked incident edge} at the end of her turn. Therefore, the third requirement demands more of Alice since she can no longer ignore any vertices of degree $n-1$.
	
	Turning to the proof, as usual, the lower bound follows from Inequality \eqref{eqn:HinG}. For the upper bound, note that vertices in $V'-V$ have degree at most $n-1$. Therefore those vertices cannot help Bob get a score of $n$ or higher. Now in $G$, Alice has a strategy that allows her to mark any vertex with $n-1$ incident marked edges at the end of her turn.
	
	Alice's strategy for the game on $G'$ is as follows. If Bob marks an edge $e \in E'-E$ with both incident vertices in $V'-V$, Alice is free to mark any vertex of $G'$. If $e$ has one incident vertex in $V$ and it is unmarked, Alice marks that vertex. Otherwise she is free to mark any other vertex. If Bob marks an edge in $E$, then Alice uses her strategy on $G$ to play in $G$, leaving no unmarked vertices with $n-1$ marked incident edges in $G$ at the end of her turn.
	
	To see this strategy forces a final score $< n$, first note that vertices in $V'-V$ cannot help Bob get a score of $n$. Now if Bob were able to get $n$ edges marked in $G'$ incident to an unmarked vertex $v\in V$, the edges cannot all be in $E$ by Alice's original strategy. Therefore at least one of the edges must be in $E'-E$. But by the strategy in the preceding paragraph, there must be exactly one such edge and it must have been marked last. However, that means that in the previous round, before this edge was marked, Bob had $n-1$ marked edges in $E$ incident to $v$. But in that case, Alice's original strategy would have marked $v$ already.
\end{proof}

\begin{corollary} \label{G'}
	$\colve(\mathcal T') = 4$.
\end{corollary}

\begin{proof}
	This follows by examining the proof in Corollary \ref{Trianlge Lattice Theorem} to see that the conditions of Theorem \ref{graph extension theorem} are met.
\end{proof}

The same technique gives the following result.

\begin{corollary}\label{Extra Lines Theorem}
	For a triangular lattice,  $D$, obtained from any of the lattices from Corollaries \ref{Trianlge Lattice Theorem}, \ref{Square Lattice Theorem}, and \ref{subway tiling theorem} by adding a vertex to the center of every triangular face and connecting it to each corner of the triangle, $\colve(D)=4$.
\end{corollary}

\section{Futher Results and Conjectures}	\label{sec: further}

Though this paper has so far focused on bounding $\colve(G)$ from above, there is no absolute upper bound for all graphs. For example, it is known that $\colve(K_n)$ is unbounded, where $K_n$ is the complete graph on $n$ vertices \cite[Theorem 17]{bresar}. Bounding the vertex-edge coloring number from below often involves
free paths and Inequality \eqref{freepath}. However, this is only useful for obtaining a lower bound of $4$ for $\colve(G)$. Higher lower bounds can be obtained by generalizing the notion of free paths.

\begin{definition}
	In a vertex-edge marking game, $\mathcal G$, on a graph $G$, an $n$-\textit{free path} is a path $P$ of $G$ with vertex sequence $v_0, \ldots, v_k$ with $k \geq 2$ so that
	\begin{itemize}
		\item the first and last edges of $P$, $v_0v_1$ and $v_{k-1}v_k$, are marked,
		\item the interior vertices, $v_1,\ldots,v_{k-1}$, are not marked, and
		\item each interior vertex is incident to at least $n+1$ edges not in $P$ and at least $n$ of those edges are marked.
	\end{itemize}
\end{definition}

Note that a $0$-free path coincides with a free path. See Figure \ref{nfreepath} for an example of a $3$-free path of length $5$.

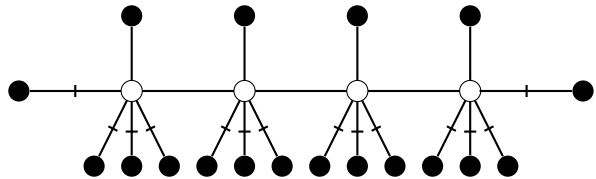
\begin{figure}[H]
    \centering
    \begin{tikzpicture}
    \tikzstyle{hollow node}=[draw,circle, fill=white, outer sep=-4pt,inner sep=0pt,
    minimum width=8pt, above]
    \node[filled, minimum width = 8pt] (0) at (0,0) {};
        \foreach \i [evaluate=\i as \j using {int(\i-1)}] in {1,...,4}{
        \node (\i) at (\i*1.5,0) {};
        \draw[black, thick] (\i) -- (\j);
        \node[filled, minimum width = 8pt] (l\i) at (\i*1.5-.5,-1){};
        \node[filled, minimum width = 8pt] (r\i) at (\i*1.5+.5,-1) {};
        \node[filled, minimum width = 8pt] (c\i) at (\i*1.5,-1) {};
        \node[filled, minimum width = 8pt] (u\i) at (\i*1.5,1) {};
        \draw[black, thick] (l\i) -- (\i);
        \draw[black, thick] (r\i) -- (\i);
        \draw[black, thick] (c\i) -- (\i);
        \draw[black, thick] (u\i) -- (\i);
        \draw[black, thick] (\i*1.5-.31,-.47) -- (\i*1.5-.19,-.53);
        \draw[black, thick] (\i*1.5+.31,-.47) -- (\i*1.5+.19,-.53);
        \draw[black, thick] (\i*1.5-.08,-.54) -- (\i*1.5+.08,-.54);

        }
        \foreach \j in {1,...,4}{
        \node[hollow node] at (\j){};
        }
        \node[filled, minimum width = 8pt] (5) at (7.5,0) {};
        \draw[black, thick] (4) -- (5);
        \draw[black, thick] (.75,-.08) -- (.75,.08);
        \draw[black, thick] (6.75,-.08) -- (6.75,.08);
    \end{tikzpicture}
    \caption{A 3-Free Path of Length $5$}
    \label{nfreepath}
\end{figure}

\begin{theorem} \label{npath}
    Let $G$ be a graph on which a vertex-edge marking game, $\mathcal G$, is being played. If there is a round in which there is an $n$-free path, $P$, and it is Bob's turn, then Bob has a strategy to force
    \[ \score(\mathcal G) \geq n+3.\]
\end{theorem}

\begin{proof}
	The proof follows immediately by induction on the length $k$ of $P$. The length $k=2$ case is trivial. For the inductive step, if the edge $v_1 v_2$ is already marked, we have a shorter $n$-path and Bob wins.
If $v_1 v_2$ is unmarked, then Bob marks $v_1 v_2$ in his first move. If Alice does not mark $v_1$, Bob marks the unmarked edge incident to $v_1$ and wins. Otherwise, we have a shorter $n$-path and Bob wins.
\end{proof}

Actually, the existence of an $n$-free path is both sufficient and necessary for $\colve(G) \geq n+4$.

\begin{corollary}
    Let $G$ be a graph. Then
    \[ \colve(G) \geq n+4\]
    if and only if Bob has a strategy able to provide him with having an $n$-free path on one of his turns.
\end{corollary}

\begin{proof}
	For the only if part, note that Alice must attempt the strategy to leave no unmarked vertices of degree $>n+2$ with $n+2$ marked incident edges at the end of her turn if she wants to avoid the final score $n+3$. This attempt can only fail if Bob has a strategy that allows him to mark the edge $v_1v_2$ between two unmarked vertices $v_1,v_2$ of degree $>n+2$ which have $n+1$ marked incident edges each. However, this means that Bob has a strategy to force an $n$-free path of length $3$ on one of his turns.

The converse direction is immediate from Theorem \ref{npath}.
\end{proof}

The focus of this paper was determining the vertex-edge coloring number for specific triangular graphs. We have a number of conjectures for related results. First, the fact that $\mathcal T$ was $2$-colorable played a major role in the proof of Corollary \ref{Trianlge Lattice Theorem}. We conjecture that, in fact, this was the key necessary hypothesis.

\begin{conj}
For any $2$-colorable planar graph $G$, $\colve(G) \leq 4$.
\end{conj}

It is natural to try to generalize many of these results to Apollonian networks. Unfortunately, we do not have a sufficiently robust strategy for Alice to permit an inductive application of Theorem \ref{graph extension theorem}. However, we still believe that the upper bound is~$4$.

\begin{conj} \label{conj}
For any Apollonian network $G$, $\colve(G) \leq 4$.
\end{conj}

It feels as if adapted techniques from \cite{Schnyder} and Inequality \eqref{eqn:outdegreebound} should give results on Apollonian networks. However, it turns out that Inequality \eqref{eqn:outdegreebound} is not, by itself, able to prove Conjecture \ref{conj}. For example, note that if $(V,E)$ is an Apollonian network with $|V|= n$ vertices, then the number of edges is $|E| = 3n-6$.
So for large $n$, $\frac{|E|}{|V|}$ tends to 3 with $\frac{|E|}{|V|}>2$ for $n>6$. Therefore, orienting the edges of the graph such that each vertex has at most two outgoing edges is not possible as the average out-degree equals $\frac{|E|}{|V|}$. With that said, switching to the dual graph on an Apollonian network may be a more useful technique.

%
%
%
%
%
%

\bibliographystyle{plain}
\bibliography{4prong}

\begin{thebibliography}{1}

\bibitem{suffice}
K.~Appel and W.~Haken.
\newblock The four color proof suffices.
\newblock {\em Math. Intelligencer}, 8:10--20, 1986.

\bibitem{bartniki}
T.~Bartnicki, B.~Bre\v{s}ar, J.~Grytczuk, M.~Kov\v{s}e, Z.~Miechowicz, and
  I.~Peterin.
\newblock Game chromatic number of {C}artesian product graphs.
\newblock {\em Electron. J. Combin.}, 15:R72, 2008.

\bibitem{complexity}
H.~L. Bodlaender.
\newblock On the complexity of some coloring games.
\newblock {\em Internat. J. Found. Comput. Sci.}, 2:133--147, 1991.

\bibitem{bresar}
B.~Bre\v{s}ar, N.~Gastineau, T.~Gologranc, and O.~Togni.
\newblock On a vertex-edge marking game on graphs.
\newblock {\em Ann. Comb.}, 25:179--194, 2021.

\bibitem{gardner}
M.~Gardner.
\newblock {\em The Last Recreations}.
\newblock Copernicus, Springer, New York, 1997.

\bibitem{mckay}
B.~D. McKay.
\newblock A note on the history of the four-colour conjecture.
\newblock {\em J. Graph Theory}, 72:361--363, 2013.

\bibitem{Schnyder}
W.~Schnyder.
\newblock Embedding planar graphs on the grid.
\newblock {\em Proc. 1st Annual ACM-SIAM Symposium on Discrete Algorithms},
  pages 138--147, 1994.

\end{thebibliography}

\end{document}